\documentclass[a4paper,10pt,aml]{elsarticle}

\usepackage{a4wide,enumerate,color,graphicx}
\usepackage{amsmath,bbm,amsfonts,amsthm}
\usepackage{scrextend} 
\allowdisplaybreaks
\usepackage[normalem]{ulem}
\usepackage[dvipsnames]{xcolor}
\usepackage{setspace}
\usepackage{amsbsy}
\usepackage{hyperref}
\usepackage[overload]{empheq}

\newcommand{\R}{{\mathbb R}} 
\newcommand{\del}{\partial}
\newcommand{\diver}{\operatorname{div}} 

\newcommand{\dx}{\textnormal{d}x}
\newcommand{\dt}{\textnormal{d}t}

\newcommand{\dy}{\textnormal{d}y}
\newcommand{\dz}{\textnormal{d}z}
\newcommand{\eps}{\epsilon}
\newcommand{\T}{\mathbb{T}^3}

\newcommand{\Z}{\mathbb{Z}}
\newcommand{\bve}{v^E}
\newcommand{\bv}{v}
\newcommand{\bvn}{v^\nu}
\numberwithin{equation}{section}
\hypersetup{
    colorlinks=true,
    linkcolor=blue,
    filecolor=blue,      
    urlcolor=blue,
}

\newtheorem{theorem}{Theorem}
\newtheorem{lemma}[theorem]{Lemma}

\newtheorem{definition}{Definition}

\begin{document}
\begin{frontmatter}
  \title
  {Three results on the
    Energy conservation for the 3D Euler equations}
  \author{Luigi C. Berselli }
  \address{Dipartimento di Matematica, Universit\`a
    di Pisa, Via F. Buonarroti 1/c, I56127, Pisa, Italy
    email: {luigi.carlo.berselli@unipi.it}}
%
%
\author{Stefanos Georgiadis}  \address{Computer, Electrical and Mathematical
    Science and Engineering Division, King Abdullah Univ. of
    Science and Technology (KAUST), Thuwal 23955-6900, Saudi Arabia
    and Institute for Analysis and Scientific Computing, Vienna
    University of Technology, Wiedner Hauptstra\ss{}e 8-10, 1040 Wien,
    Austria email: {stefanos.georgiadis@kaust.edu.sa}}
%
%
  \date{\today}
  \begin{abstract}
    We consider the 3D Euler equations for incompressible homogeneous
    fluids and we study the problem of energy conservation for weak
    solutions in the space-periodic case. First, we prove the energy
    conservation for a full scale of Besov spaces, by extending some
    classical results to a wider range of exponents. Next, we consider
    the energy conservation in the case of conditions on the gradient,
    recovering some results which were known, up to now, only for the
    Navier-Stokes equations and for weak solutions of the Leray-Hopf
    type. Finally, we make some remarks on the Onsager singularity
    problem, identifying conditions which allow to pass to the limit
    from solutions of the Navier-Stokes equations to solution of the
    Euler ones, producing weak solutions which are energy conserving.
  \end{abstract}
\end{frontmatter}
\section{Introduction}
The aim of this paper is to extend some nowadays classical results
about the energy conservation for the space-periodic 3D Euler
equations (here $\T:=\big(\R/2\pi\Z\big)^{3}$)
\begin{equation}
  \label{eq:Ebvip}
  \begin{aligned}
    \partial_{t}\bve+(\bve\cdot\nabla)\,\bve+\nabla p^{E}&=0\qquad
    &\text{ in } (0,T)\times \T,
    \\
    \diver\bve&=0\qquad &\text{ in }(0,T)\times \T,
    \\
    \bve(0)&=\bve_0 \qquad & \text{in }\T,
  \end{aligned}
\end{equation}
to embrace a full space-time range of exponents.  Recall that it is known
since~\cite{CET1994} that weak solutions of the Euler equations such that
\begin{equation}
  \label{eq:CET}
  \bve\in C_{w}(0,T;L^{2}(\T))\cap L^{3}(0,T;B^{\alpha}_{3,\infty}(\T)),\qquad\text{with}\quad
  \alpha>\frac{1}{3}, 
\end{equation}
conserve the energy, where $B^{\alpha,\infty}_{3}(\T)$ denotes a
standard Besov space and the motivation for this result is the Onsager
conjecture~\cite{Ons1949} following from Kolmogorov K41 theory.  The
Onsager conjecture (only recently solved also for the negative part,
see Isett~\cite{Ise2018} and De Lellis~\cite{BDLSV2019}) suggested the
threshold value $\alpha=1/3$ for energy conservation. Here, we
consider a combination of space-time conditions, identifying families
of Besov spaces with the range $(1/3,1)$ for the exponent of
regularity balanced by a proper integrability exponent in time. Next,
we consider also the limiting case $\alpha=1$, and finally the
connection of energy conservation with the vanishing viscosity
limits. We recall that the first rigorous results about Onsager
conjecture are probably those of Eyink~\cite{Eyi1994,Eyi1995} in the
Fourier setting and Constantin, E, and Titi~\cite{CET1994} and we are
mainly inspired by these references; for the vanishing viscosity limit
we follow the same path as in Drivas and Eyink~\cite{DE2019}.

The original results we prove concern: 1) the extension
of~\eqref{eq:CET} to a full scale of exponents $\frac{1}{3}<\alpha<1$,
identifying the sharp conditions on the parameters, as previously done
in the setting of H\"older continuous functions in~\cite{Ber2023}; 2)
the extension to the case $\alpha=1$, which means that we look for
conditions on the gradient of $\bve$ in standard Lebesgue spaces; 3)
we identify hypotheses, uniform in the viscosity, on solutions to the
Navier-Stokes equations, which allow us to pass to the limit as
$\nu\to0$ and to construct weak solutions of the Euler equations
satisfying the energy equality.

\medskip

More precisely, concerning point 1) we extend the result
of~\cite{CET1994} to a wider range of exponents proving the following
theorem:
\begin{theorem}
  \label{thm:1}
  Let $\bve$ be a weak solution to the Euler equations such that, for
  $\frac{1}{3}<\alpha<1$, 
  \begin{equation}
    \label{eq:thm1}
\bve\in    L^{1/\alpha}(0,T;B^\beta_{\frac{2}{1-\alpha},\infty}(\T)),
    \qquad\text{with}\quad\alpha<\beta<1.
\end{equation} 
Then, $\bve$ conserves the energy.
\end{theorem}
Similar results have already been proved in the setting of H\"older continuous functions
(see~\cite{Ber2023b}) and in both cases, one can see that the limiting case
$\alpha\to1^{-}$ leads formally to $L^{1}(0,T;W^{1,\infty})$, which corresponds to the
Beale-Kato-Majda criterion. Anyway, working directly with the velocity $\bve$ in a Sobolev
space, we obtain the following result:
\begin{theorem}
  \label{thm:2}
  Let $\bve$ be a weak solution to the Euler equations such that, for
  $q>2$,
  \begin{equation}
    \label{eq:thm2}
 \bve\in  L^{r}(0,T;W^{1,q}(\T)),\qquad \text{with}\quad r>\frac{5q}{5q-6}.
\end{equation}
Then, $\bve$ conserves the energy.
\end{theorem}
The sharpness of this result comes by observing that we recover for the Euler equations
the same results (at least in this range of exponents) which are known for Leray-Hopf weak
solutions to the Navier-Stokes equations, see the recent results in~\cite{BY2019,BC2020}.

\medskip

Concerning point 3) we extend results from~\cite{DE2019,Ber2023} on
the emergence of solutions to the Euler equations satisfying the
energy equality as inviscid limits of Leray-Hopf weak solutions to the
Navier-Stokes equations ($\nu>0$)
\begin{equation}
  \label{eq:NSEbvip}
  \begin{aligned}
    \partial_{t}\bvn+(\bvn\cdot\nabla)\,\bvn-\nu\Delta\bvn+\nabla
    p^{\nu}&=0\qquad &\text{ in }(0,T)\times \T,
    \\
    \diver\bvn&=0\qquad &\text{ in }(0,T)\times \T,
    \\
    \bvn(0 )&=\bvn_0 \qquad & \text{ in }\T.
  \end{aligned}
\end{equation}

This result generalizes to a wider range of exponents the result
from~\cite{DE2019} which deals with $\alpha\sim 1/3$ and also the
results in \cite{Ber2023}, which are in the setting of H\"older
continuous functions, but with a more restrictive time-dependence for
$\alpha>1/2$. We have the following result:
\begin{theorem}
\label{thm:3}
Let $\{\bvn\}_{\nu>0}$ be a family of weak solutions of the NSE with
the same initial datum $v_{0}\in H\cap B^\beta_{\frac{2}{1-\alpha},\infty}(\T)$, $\beta>\alpha$.   Let us also 
assume that for   $\frac{1}{3}<\alpha<1$, and $\alpha<\beta<1$
there exists a constant $C_{\alpha,\beta}>0$, independent of $\nu>0$,
such that
  \begin{equation}
    \label{eq:1}
    \|\bvn\|_{L^{1/\alpha}(0,T; B^\beta_{\frac{2}{1-\alpha},\infty})}\leq
    C_{\alpha,\beta},\qquad\forall\,\nu\in(0,1].
  \end{equation}
  Then, in the limit $\nu\to0$ the family $\{v^{\nu}\}$ converges (up
  to a sub-sequence) to a weak solution $\bve$ in $[0,T]$ of the Euler
  equations satisfying the energy equality. The same holds if
  alternatively $v_{0}\in H\cap W^{1,q}(\T)$, with $q>2$, and
  $\|\bvn\|_{L^{r}(0,T; W^{1,q}(\T))}\leq C_{r,q}$, for all
  $\nu\in(0,1]$ and for
  $r>\frac{5q}{5q-6}$.
\end{theorem}
The problem of vanishing viscosity and construction of distributional
(dissipative) solutions to the Euler equations has a long history and
we mainly refer to Duchon and Robert~\cite{DR2000} for similar
results. We also wish to mention the Fourier-based approach recently
developed by Chen and Glimm~\cite{CG2012,CG2019} where spectral
properties are used to deduce certain fractional regularity results,
suitable to prove the inviscid limit. Our proof uses standard
mollification and handling of the commutation terms. Even though the
results use elementary techniques, they are new and rather sharp. Note
also that Theorem~\ref{thm:3} implies that ``quasi-singularities'' are
required in Leray-Hopf weak solutions in order to account for
anomalous energy dissipation, see the discussion and interpretation
in~\cite{DE2019,Ber2023}. Moreover, due to results
in~\cite{BC2020, BY2019}, weak solutions of the NSE with
$\nabla \bvn\in L^{r}(0,T;L^{q}(\T))$, $r>\frac{5q}{5q-6}$ conserve the
energy and the same holds for the limiting solutions of the Euler equations.

\smallskip

\textbf{Plan of the paper:} In Section~\ref{sec2} we set up our
notation by giving the definitions of the spaces and the solutions
that we use throughout the paper. Moreover, we recall the basic
properties of the mollification and the commutator formula that will
be used extensively in the proofs of the theorems. In
Section~\ref{sec3} we give the proofs of Theorem~\ref{thm:1}
and~\ref{thm:2}, investigating minimum regularity conditions for
energy conservation, for weak solutions to the Euler
equations. Finally, in Section~\ref{sec4}, we give the proof of
Theorem~\ref{thm:3}, dealing with the emergence of weak solutions of
Euler in the limiting case $\nu\to0$.
%
%


\section{Notation}\label{sec2}
In the sequel we will use the Lebesgue
$(L^{p}(\T),\|\,.\,\|_{p})$ and Sobolev
$(W^{1,p}(\T),\|\,.\,\|_{1,p})$ spaces, with $1\leq p\leq\infty$; for
simplicity we denote by $(\,.\,,\,.\,)$ and $\|\,.\,\|$ the $L^{2}$ scalar product and
norm, respectively, while the other norms are explicitly indicated.  By $H$ and $V$ we denote the
closure of smooth, periodic, and divergence-free vector fields in $L^{2}(\T)$ or
$W^{1,2}(\T)$, respectively. Moreover, we will use the Besov spaces
$B^{\alpha}_{p,\infty}(\T)$, which are the same as {N}ikol'ski\u{\i} spaces
$\mathcal{N}^{\alpha,p}(\T)$. They are sub-spaces of $L^{p}$ for which there exists $c>0$, such that
$\|u(\cdot+h)-u(\cdot)\|_{p}\leq c |h|^{\alpha}$, and the smallest constant is the semi-norm $[\,.\,]_{B^{\alpha}_{p,\infty}}$.

To properly set the problem we consider, we give the definitions of weak solutions:
  \begin{definition}[Weak solution to the Euler equations]
  \label{def:Euler-weak-solution}
Let $\bv_0\in H$. A measurable function
  $\bve:\,(0,T)\times \T\to \R^3$ is called a  weak
  solution to the Euler equations~\eqref{eq:Ebvip} if $
  \bve\in L^\infty(0,T;H)$, 
  solves the equations in the weak sense:
\begin{equation}
  \label{eq:distributional-solution2E}
  \int_0^{T}\int_{\T}\Big[\bve\cdot\partial_{t}\phi+
  (\bve\otimes\bve):\nabla\phi\Big]\,\dx\,d
  t=-\int_{\T}\bv_0\cdot\phi(0)\,\dx, 
\end{equation}
for all $\phi\in \mathcal{D}_T:=\Big\{\phi\in C^{\infty}_{0}([0,T[\times\T):\
  \diver\phi=0\Big\}$.
\end{definition}
We also recall the definition of weak solutions to the Navier-Stokes equations. 
\begin{definition}[Space-periodic Leray-Hopf weak solution]
  \label{def:LH-weak-solution}
  Let $\bv_0^{\nu}\in H$. A
  measurable function $\bvn:\,(0,T)\times \T\to \R^{3}$ is called a
  Leray-Hopf weak solution to the space-periodic
  NSE~\eqref{eq:NSEbvip} if
  $ \bvn\in L^\infty(0,T;H)\cap L^2(0,T;V)$ and the following hold
  true:
  \\
The function  $\bvn$ solves the equations in the weak sense:
\begin{equation}
  \label{eq:distributional-solution2}
  \int_0^{T}\int_{\T}\Big[\bvn\cdot\partial_{t}\phi
  -\nu\nabla\bvn:\nabla\phi+
  (\bvn\otimes\bvn):\nabla\phi\Big]\,\dx\dt=
  -\int_{\T}\bvn_0\cdot\phi(0)\,\dx,
\end{equation}
for all $\phi\in \mathcal{D}_T$;
\\
The (global) energy inequality holds:
\begin{equation}
  \label{eq:energy_inequality}
  \frac{1}{2} \|\bvn(t)\|^2_{2}+\nu\int_0^t\|\nabla\bvn(\tau)\|^2_{2}\,
  d \tau\leq\frac{1}{2}\|\bvn_0\|^2_{2}
  ,\qquad\forall\,t\in[0,T];
\end{equation}
\\
The initial datum is strongly attained:
  $\lim_{t\to0^+}\|\bvn(t)-\bvn_0\|=0$.
\end{definition}
%
%
%
%
%
\subsection{Mollification and Sobolev/Besov spaces}
We use the classical tools of mollification to justify calculations
and to this end we fix 
$\rho\in C^{\infty}_{0}(\R^{3})$ such that
$\rho(x)=\rho(|x|),\   \rho\geq0,\ \text{supp }\rho\subset
  B(0,1)\subset\R^3,\ \int_{\R^{3}}\rho(x)\,\dx=1$,  
 and we define, for $\epsilon\in(0,1]$, the Friedrichs family 
$\rho_{\epsilon}(x):=\epsilon^{-3}\rho(\epsilon^{-1}x)$. Then, for any function
$f\in L^{1}_{loc}(\R^{3})$ we define by the usual convolution
\begin{equation*}
  f_{\epsilon}(x):=\int_{\R^{3}}\rho_{\epsilon}(x-y)f(y)\,\dy=\int_{\R^{3}}
  \rho_{\epsilon}(y)f(x-y)\,\dy. 
\end{equation*}
If $f\in L^{1}(\T)$, then $f\in L^{1}_{loc}(\R^{3})$, and it turns out that
$f_{\epsilon}\in C^{\infty}(\T)$ is $2\pi$-periodic along the $x_{j}$-direction, for
$j=1,2,3$.
Moreover, if $f$ is a divergence-free vector field, then
$f_{\epsilon}$ is a smooth divergence-free vector field. We report now
the basic properties of the convolution operator we will use in the
sequel, see for instance~\cite{CET1994,Ber2021,Ber2023}.
\begin{lemma}
  \label{lem:convolution-Holder}
Let $\rho$ be as above. If $u\in L^{q}(\T)$, then $\exists\, C>0$
  (depending only on $\rho$) such that
  \begin{equation}
    \label{eq:conv7}
            \|u_{\epsilon}\|_{r}\leq\frac{C}{\epsilon^{3(\frac{1}{q}-\frac{1}{r})}}
                      \|u\|_{q}\quad \text{for all }r\geq q;
  \end{equation}
  If   $u\in B^{\beta}_{q,\infty}(\T)$, then 
    \begin{align}
      \label{eq:conv1}      &   
                              \|u(\cdot +y )-u(\cdot )\|_{q}\leq
                              [u]_{B^{\beta}_{q,\infty}}|y |^{\beta},
      \\
      \label{eq:conv2}      &   \|u-u_{\epsilon}\|_{q}\leq
                              [u]_{B^{\beta}_{q,\infty}}\, \epsilon^{\beta}, 
      \\
      \label{eq:conv3}      &\|\nabla u_{\epsilon}\|_{q}\leq C
                              [u]_{B^{\beta}_{q,\infty}}\,\epsilon^{\beta-1}, 
    \end{align}
    while if $u\in W^{1,q}(\T)$, then 
   \begin{align}
    \label{eq:conv4}      &   
                              \|u(\cdot +y )-u(\cdot)\|_{q}\leq
                              \|\nabla u\|_{q}|y |,
      \\
      \label{eq:conv5}      &   \|u-u_{\epsilon}\|_{q}\leq
                              \|\nabla u\|_{q}\, \epsilon, 
      \\
      \label{eq:conv6}      &\|\nabla u_{\epsilon}\|_{q}\leq C
                              \|u\|_{q}\,\epsilon^{-1}.
    \end{align}
  \end{lemma}
  In the sequel, the following well-known commutator formula derived
  in~\cite{CET1994} and known as the ``Constantin-E-Titi commutator'' will be crucial:
\begin{equation}
  \label{eq:commutator-Euler}
  (u\otimes u)_{\epsilon}=u_{\epsilon}\otimes u_{\epsilon}
  +r_{\epsilon}(u,u)-(u-u_{\epsilon})\otimes(u-u_{\epsilon}) ,  
\end{equation}
with
\begin{equation*}
  r_{\epsilon}(u,u):=\int_{\T}\rho_{\epsilon}(y)(\delta_{y}u(x)
  \otimes\delta_{y}u(x))\,\dy,\qquad\text{for}\quad \delta_{y}u(x):=u(x-y)-u(x).
\end{equation*}
%


\section{On the conservation of energy for ideal fluids}\label{sec3}
We prove Theorem~\ref{thm:1} and,
for $\beta\in(\frac{1}{3},1)$, we investigate the minimum Besov
regularity that is needed, so that weak solutions of the Euler
equations conserve their kinetic energy.

  \begin{proof}[Proof of Theorem~\ref{thm:1}]
    We test the Euler equations against
    $\varphi=(\bve_\epsilon)_\epsilon$ to obtain \begin{equation}
      \begin{split}\label{energy-eq}
        \frac{1}{2}\|\bve_\epsilon(T)\|_{L^2(\Omega_T)}^2 &
        =\frac{1}{2}\|\bve_\epsilon(0)\|_{2}^2-\int_0^T\int_{\T}
        r_\epsilon(\bve,\bve):\nabla \bve_\epsilon\,\dx\dt
        \\
        & +
        \int_0^T\int_{\T}(\bve-\bve_\epsilon)\otimes(\bve-\bve_\epsilon):\nabla
        \bve_\epsilon\,\dx\dt,
      \end{split}
    \end{equation}
    since
    $ \int_0^T\int_{\T} \bve_\epsilon\otimes \bve_\epsilon:\nabla \bve_\epsilon\,\dx\dt=0$,
    due to $\bve_{\epsilon}$ being smooth and divergence-free.  We now estimate the last two
    terms from the right-hand side, using the properties of Besov spaces.
%
%
Indeed, for $0<\eta<2$ and $q>1$, such that $q-(1+\eta)>0$ we write:
\[ \begin{split} I_1 &
    :=\left|\int_0^T\int_{\T}(\bve-\bve_\epsilon)\otimes(\bve-\bve_\epsilon):\nabla
      \bve_\epsilon\,\dx\dt\right|
    \\
    & \leq
    \int_0^T\int_{\T}|\bve-\bve_\epsilon|^\eta|\bve-\bve_\epsilon|^{2-\eta}|\nabla
    \bve_\epsilon|\,\dx\dt
    \\
    & \leq \int_0^T\|\bve-\bve_\epsilon\|_{q}^\eta\,
    \|\bve-\bve_\epsilon\|_{\frac{(2-\eta)q}{q-(1+\eta)}}^{2-\eta}\|\nabla
    \bve_\epsilon\|_{q}\,\dt,
      \end{split} \]
    and in the second line we used H\"older's inequality. Since a weak solution
    $\bve$ is in $L^\infty(0,T;H)$, if $\frac{(2-\eta)q}{q-(1+\eta)}=2$, we can use the energy
    bound to infer
    \begin{equation*}
      \|\bve-\bve_\eps\|_{2}\leq\|\bve\|_{2}+\|\bve_\eps\|_{2}\leq
      2\|\bve\|_{2}\leq2\,\textnormal{esssup}_{t\in(0,T)}\|\bve\|_{2}\leq C,
    \end{equation*}
    and thus by~\eqref{eq:conv2}-\eqref{eq:conv3}
    \[ \begin{split}
        I_1 & \leq C\int_0^T\|\bve-\bve_\eps\|_{q}^\eta\,\|\nabla
        \bve_\eps\|_{q}\,\dt
         \leq
        C\eps^{\beta\eta+\beta-1}\int_0^T[\bve(t)]_{B_{q,\infty}^\beta}^{\eta+1}\dt,
      \end{split} \] where $C>0$ does not depend on $\eps>0$.

    Next, we estimate the remainder term in the commutator as follows:
    \[ \begin{split}
      r_\epsilon(\bve,\bve) & =
      \int_{\T}\rho_\epsilon(y)(\bve(x-y)-\bve(x))\otimes(\bve(x-y)-\bve(x))\,\dy
      \\
      & \overset{y=\eps z}{=} \int_{\T}\rho(z)(\bve(x-\epsilon
      z))-\bve(x))\otimes(\bve(x-\epsilon 
      z))-\bve(x))\,\dz 
      \\
      & \leq \int_{\T}|\bve(x-\epsilon z)-\bve(x)|^2\,\dz.
    \end{split} \] 
Then, as above, we can write for $0<\eta<2$ such that $\frac{(2-\eta)q}{q-(1+\eta)}=2$: 
\[ \begin{split} I_2 & := \left|\int_0^T\int_{\T} r_\epsilon(\bve,\bve):\nabla
    \bve_\epsilon\,\dx\dt\right|
  \\
  & \leq \int_0^T\int_{\T}\int|\bve(x-\epsilon z)-\bve(x)|^2|\nabla \bve_\eps|\,\dz\dx\dt
  \\
  & = \int_0^T\int_{\T}\int_{\T}|\bve(x-\epsilon z)-\bve(x)|^\eta|\bve(x-\epsilon
  z)-\bve(x)|^{2-\eta}|\nabla \bve_\eps|\,\dz\dx\dt
  \\
  & \leq C\int_0^T\int_{\T}\|\nabla \bve_\eps\|_{q}\,\|\bve(\cdot-\eps
  z)-\bve(\cdot)\|_{q}^\eta\,\|\bve(\cdot-\eps
  z)-\bve(\cdot)\|_{\frac{(2-\eta)q}{q-(1+\eta)}}^{2-\eta}\,\dx\dt.
      \end{split} \] 
    and using~\eqref{eq:conv2}-\eqref{eq:conv3} we arrive at
 \[ \begin{split}
        I_2 & \leq
        C\eps^{\beta\eta+\beta-1}\int_0^T\int_{\T}|z|^{\eta\beta}[\bve]_{B_{q,\infty}^\beta}^{\eta+1}\,\dz\dt  
 \leq
        C\eps^{\beta\eta+\beta-1}\int_0^T\|\bve\|_{B_{q,\infty}^\beta}^{\eta+1}\dt,
      \end{split} \] with $C>0$ independent of $\eps$.

    Hence, for fixed $\beta\in(\frac{1}{3},1)$, we want to find
    $(q,\eta)\in(1,+\infty)\times(0,2)$ such that $\eta+1$ (the
    exponent of the Besov semi-norm) is the smallest possible, subject
    to the following set of constraints:
    \begin{equation*}
    \left\{\begin{array}{cc}
             \beta\eta+\beta-1>0
             \\
             \eta<q-1
             \\
             \frac{(2-\eta)q}{q-(1+\eta)}=2
           \end{array}\right..
       \end{equation*}
       Hence, we set $\eta=\frac{1-\alpha}{\alpha}$ and
       $q=\frac{2}{1-\alpha}$. The last two constraints are satisfied
       for $\alpha>1/3$ (leading to $q>3$); next, if
       $\beta>\alpha$ 
       , then
       $\beta\eta+\beta-1=\frac{\beta-\alpha}{\alpha}>0$ and consequently
%
       \[
         0\leq I_1+I_2\leq
         C\epsilon^\frac{\beta-\alpha}{\alpha}\int_0^T[\bve]^{\beta}_{B_{\frac{2}{1-\alpha},\infty}^\beta}\dt\overset{\epsilon\to0}{\longrightarrow}0
       \] 
and letting $\epsilon\to0$ in~\eqref{energy-eq} gives
\begin{equation} \label{energy-equation}
  \frac{1}{2}\|\bve(t)\|_{L^2}^2
  =\frac{1}{2}\|\bve(0)\|_{L^2}^2.
 \end{equation}
 Therefore, for $\alpha\in(\frac{1}{3},1)$, the ``critical'' space for
 energy conservation is
 $L^{1/\alpha}(0,T;B^\alpha_{\frac{2}{1-\alpha},\infty})$.
 \end{proof}

\medskip                                             

We now prove the second theorem, corresponding to conditions on the gradient of $\bve$,
which would, formally, be the same with $\alpha=1$, but in fact the result here is much
stronger, since the bound on the gradient allows us to make sense of the convective term
in a more precise manner.
\begin{proof}[Proof of Theorem~\ref{thm:2}]
  In the case $\alpha=1$, the required regularity for energy
  conservation is $\nabla u\in L^{r}(0,T;L^q(\Omega))$, for
  $r>\frac{5q}{5q-6}$, as follows from the following computations. The
  approach is similar as before and we just need to control the
  commutator terms, after testing the equations by $(\bve_\eps)_\eps$
  (we make explicit computations only for this one, since the
  remainder can be handled as we have done before). We get in fact
  \begin{align*}
    I_1:&= \left|\int_0^T\int_{\T}(\bve-\bve_\eps)\otimes(\bve-\bve_\eps):\nabla
      \bve_\eps\,\dx\dt\right|
    \\
    &\leq \int_0^T\int_{\T}|\bve-\bve_\eps|^2|\nabla \bve_\eps|\,\dx\dt
    \\
    & \leq \int_0^T\|\bve-\bve_\eps\|^2_{2p}\,\|\nabla \bve_\eps\|_{p'}\,\dx\dt
    \\
    & \leq
    \int_0^T\|\bve-\bve_\eps\|_{2}^{2\theta}\,\|\bve-\bve_\eps\|_{q}^{2(1-\theta)}\|\nabla
    \bve_\eps\|_{p'}\,\dx\dt,
    \end{align*} 
    where in the second step we used H\"older's inequality with
    conjugate exponents $p$ and $p'$ (to be determined), and in the
    third one convex interpolation such that
    $\frac{1}{2p}=\frac{\theta}{2}+\frac{1-\theta}{q}$, with
    $2p<q$. Now, using the fact that $\bve\in L^\infty(0,T;L^2(\T))$
    and inequality~\eqref{eq:conv7} for the gradient of $\bve$
    \[ \|\nabla
      \bve_\eps\|_{p'}\leq C\eps^{-3\left(\frac{1}{q}-\frac{1}{p'}\right)}\|\nabla
      \bve\|_{q}, \qquad\text{for }p'>q,
      \] we obtain:
    \begin{align*} I_1&\leq
      C\eps^{-3\left(\frac{1}{q}-\frac{1}{p'}\right)}\int_0^T\|\bve-\bve_\eps\|_{q}^{\frac{2q(p-1)}{p(q-2)}}\|\nabla
      \bve\|_{q}\,\dt \end{align*}
    and \eqref{eq:conv5}-\eqref{eq:conv6} yield:
    \begin{align*} I_1\leq
      C\eps^{\frac{2q(p-1)}{p(q-2)}-3\left(\frac{1}{q}-\frac{1}{p'}\right)}\int_0^T\|\nabla
      \bve\|_{q}^{\frac{2q(p-1)}{p(q-2)}+1}\,\dt.
    \end{align*}
    Now, we want to choose $p$ such that the exponent of $\eps$ is non-negative,
    corresponding to
    \begin{equation*}
      p>\frac{(5q-6)q}{5q^2-9q+6},
\end{equation*}
and notice that the expression on the denominator is strictly
positive. Moreover, the constraints $2p<q$ and $p'>q$
imply that $p<\min\left\{\frac{q}{2},\frac{q}{q-1}\right\}$ and thus
the range of allowed $p$ is
 \begin{equation*}
\frac{(5q-6)q}{5q^2-9q+6}<p<\min\left\{\frac{q}{2},\frac{q}{q-1}\right\}.
\end{equation*}
The second term on the right-hand side of~\eqref{energy-eq} is handled
the same way and yields the same range for $p$. Then, for the infimum
value of $p$ that makes the exponent of $\eps$
non-negative, we get that the
exponent $r$ of the $L^q$ norm of the gradient becomes
$r=\frac{5q}{5q-6}$ and thus the ``critical'' space for energy conservation
(in the case $\alpha=1$) is
$\nabla \bve\in L^{\frac{5q}{5q-6}}(0,T;L^q(\T))$.
\end{proof}


\section{Inviscid limit from Navier-Stokes to Euler}\label{sec4}
In this section we prove Theorem~\ref{thm:3}, dealing with the inviscid (singular)
limit $\nu\to0$ and identifying sufficient conditions to construct weak solutions of the
Euler equations conserving the kinetic energy. 
%
%
\begin{proof}[Proof of Theorem~\ref{thm:3}]
We give the proof only in the Besov case, since the other one is pretty similar.  In the weak formulation~\eqref{eq:distributional-solution2} of the NSE
we set $\varphi=(v_\epsilon^\nu)_\epsilon$.
Note that since $v^\nu_0\in L^2(\Omega)$, we deduce, being $v^\nu$ a
Leray-Hopf solution, that
\begin{equation*}
\|u^\nu\|_{L^\infty(0,T;L^2(\T))}+ \|\sqrt{\nu}\nabla u^\nu\|_{L^2((0,T)\times\T)}\leq C.
\end{equation*}
 Moreover,
since
$v^\nu\in L^{1/\alpha}(0,T;B^\beta_{2/(1-\alpha),\infty})$,
it has a derivative in the sense of distributions in the space
$L^{1/2\alpha}(0,T;B_{1/(1-\alpha),\infty}^{\beta-2})$, given by
%
$\frac{dv^\nu}{dt}=-\mathbb{P}\diver (v^\nu\otimes
 v^\nu)+\nu\Delta v^\nu$, where $\mathbb{P}$ is the Leray projector.
%
Indeed, by comparison,  (the subscript ``$\sigma$'' means divergence-free)
\begin{equation*} 
\big\langle\int_0^T\del_tv^\nu,\phi \,\dt\big\rangle=\big\langle\int_0^T-\diver(v^\nu\otimes
v^\nu)+\nu\Delta u^n,\phi \,dt\big\rangle,\qquad\forall\,\phi\in 
 C_{\sigma}^\infty(\T),
\end{equation*} {where $\langle\cdot,\cdot\rangle$ is the duality pairing between elements of $\mathcal{D}^*(\T)$ and $\mathcal{D}(\T)=C^\infty(\T)$.}
 Choosing $\phi(x,t)=\psi(t)\varphi(x)$, with
$\psi\in C_{0,\sigma}^\infty(0,T)$ and
$\varphi\in C_{\sigma}^\infty(\T)$ we obtain
\[
\big\langle\int_0^T\del_tv^\nu\psi,\varphi\,\dt\big\rangle=\int_0^T\psi(t)\big\langle\left[-\mathbb{P}\diver
  (v^\nu\otimes v^\nu)+\nu\Delta u^n\right],\varphi\big\rangle\,\dt
\]
and note that
%
\[
  \begin{split} \|\mathbb{P}\diver (v^\nu\otimes
  v^\nu)\|_{L^{1/2\alpha}(0,T;B_{1/(1-\alpha),\infty}^{\beta-2})} & \leq C\|v^\nu\otimes
  v^\nu\|_{L^{1/2\alpha}(0,T;B_{1/(1-\alpha),\infty}^{\beta-1})}
  \\
  & {\leq C\|v^\nu\|^2_{L^{1/\alpha}(0,T;B_{2/(1-\alpha),\infty}^{\beta-1})} }
  \\
  & \leq C\|v^\nu\|^2_{L^{1/\alpha}(0,T;B_{2/(1-\alpha),\infty}^{\beta})}.
  \end{split} \]
Moreover, since $\T$ is bounded and $T<+\infty$, the embeddings
  \[
    L^{1/\alpha}(0,T;B_{2/(1-\alpha),\infty}^{\beta-2})\hookrightarrow L^{1/2\alpha}(0,T;B_{2/(1-\alpha),\infty}^{\beta-2})\hookrightarrow
  L^{1/2\alpha}(0,T;B_{1/(1-\alpha),\infty}^{\beta-2}),
\] are continuous, thus
 \[ \begin{split}
   \|\Delta v^\nu\|_{L^{1/2\alpha}(0,T;B_{1/(1-\alpha),\infty}^{\beta-2})} & \leq
   C\|\Delta v^\nu\|_{L^{1/2\alpha}}(0,T;B_{2/(1-\alpha),\infty}^{\beta-2}) 
   \\
   & \leq C\|\Delta v^\nu\|_{L^{1/\alpha}}(0,T;B_{2/(1-\alpha),\infty}^{\beta-2})
   \\
   &\leq
 C\|v^\nu\|_{L^{1/\alpha}(0,T;B_{2/(1-\alpha),\infty}^{\beta})}, 
\end{split} 
\]
and we conclude that $\del_tv^\nu\in L^{1/2\alpha}(0,T;B_{1/(1-\alpha),\infty}^{\beta-2})$.

Therefore, by the Aubin-Lions' lemma, there exists a sub-sequence
(which is not relabeled) such that:
\[
\begin{aligned}
  v^\nu&\to v ~\textnormal{strongly in}~ L^q(0,T;H) \quad\forall\, q\in(1,\infty)
\\
  \sqrt{\nu}\nabla v^\nu &\rightharpoonup 0 ~\textnormal{weakly in}~ L^2(0,T;H) 
\\
  \del_t v^\nu&\rightharpoonup\del_t v ~\textnormal{weakly in}~
  L^{1/2\alpha}(0,T;B_{1/(1-\alpha),\infty}^{\beta-2}),
\end{aligned}
\]
 which is enough to pass to the limit as $\nu\to0$
in~\eqref{eq:distributional-solution2}, proving that $v$ is a solution of the Euler equations.

As a final step, we show that the dissipation in the
energy equation goes to zero as $\nu\to0$, yielding an energy equation for the
limiting solution $v$, as well. Indeed, for $\beta>\alpha$ and
$\alpha\in(\frac{1}{3},\frac{1}{2}]$:
\[ \begin{split} \nu\int_0^T\|\nabla v_{\eps}^\nu\|_{2}^2\,\dt & \leq
    C\nu\int_0^T\|\nabla v_{\eps}^\nu\|_{\frac{2}{1-\alpha}}^2\,\dt
    \leq
    C\nu\eps^{2(\beta-1)}\int_0^T\|v^\nu\|_{B^\beta_{\frac{2}{1-\alpha},\infty}}^2\,\dt
    \\
    &
    \leq
    C\nu\eps^{2(\beta-1)}\|v^\nu\|_{L^\frac{1}{\alpha}(0,T;B^\beta_{\frac{2}{1-\alpha},\infty})}^2,
\end{split} \]
where $C>0$ does not depend on $\eps$ or $\nu$ and we need to choose
$\nu$ going to zero faster than $\eps^{2(1-\beta)}$. This is the
extension of the result in~\cite{Ber2023} for the H\"older case. In
the case $\alpha>1/2$ we prove here a slightly better result, since
\[
  \begin{split}
    \nu\int_0^T\|\nabla v_{\eps}^\nu\|_{2}^2\,\dt & =
    \nu\int_0^T\|\nabla v_{\eps}^\nu\|_{2}^{2-\frac{1}{\beta}}\|\nabla
    v_{\eps}^\nu\|_{2}^{\frac{1}{\beta}}\,\dt
    \leq
    \nu\int_0^T\left(\frac{1}{\eps}\|v_{\eps}\|_{2}\right)^{2-\frac{1}{\beta}}\left(\eps^{\beta-1}
       \|v_{\eps}\|_{B^\beta_{\frac{2}{1-\alpha},\infty}}\right)^{\frac{1}{\beta}}\,\dt 
    \\
    & \leq
    C\nu\eps^{-1}\|v^\nu\|_{L^\frac{1}{\alpha}(0,T;B^\beta_{\frac{2}{1-\alpha},\infty})}^2,
  \end{split}
\]
with $C>0$ independent of $\eps$ and $\nu$. One needs to choose
$\nu$ going to zero faster than $\eps$.

So, in the  case $\frac{1}{3}<\beta\leq\frac{1}{2}$ we have 
%
 $ \nu\int_0^T\|\nabla v_{\eps}^\nu\|_{2}^2\,\dt=O(\nu\eps^{2(\beta-1)})$,
%
and thus
\[
  \frac{1}{2}\|v^\nu(T)\|_{2}^2-\frac{1}{2}\|u_0^\nu\|_{2}^2=O(\eps^{\frac{\beta-\alpha}{\alpha}})
  +O(\nu\eps^{2(\beta-1)});
\]
on the other hand in the case $\frac{1}{2}<\beta<1$ we have
$\nu\int_0^T\|\nabla v_{\eps}^\nu\|_{2}^2\,\dt=O(\nu\eps^{-1})$,
%
 and thus
\[
  \frac{1}{2}\|v^\nu(T)\|_{2}^2-\frac{1}{2}\|u_0^\nu\|_{2}^2=O(\eps^{\frac{\beta-\alpha}{\alpha}})
  +O(\nu\eps^{-1}).
\]
Since $\eps>0$ is arbitrary, we can optimize the upper bound, the same way it was
performed in~\cite{DE2019}, by balancing the contribution of the nonlinear flux with the
one of the dissipation. Choosing $\eps\sim\nu^{\alpha/(\alpha+\beta-2\alpha\beta)}$ in the first case and
$\eps\sim\nu^{\alpha/\beta}$ in the second one, yields the upper bounds
\[
  \frac{1}{2}\|v^\nu_{\eps}(T)\|_{2}^2-\frac{1}{2}\|u_0^\nu\|_{2}^2=O({\nu^\frac{\beta-\alpha}{\alpha-2\alpha\beta+\beta}
}),
\]
and
\[
  \frac{1}{2}\|v^\nu_{\eps}(T)\|_{2}^2-\frac{1}{2}\|u_0^\nu\|_{2}^2=O({\nu^\frac{\beta-\alpha}{\beta}}),
  \]
respectively, hence showing that as $\eps,\nu\to0$ with the above rates, the kinetic
energy is conserved.
\end{proof}
\section*{Acknowledgments}
LCB acknowledges support by MIUR, within
project PRIN20204NT8W4: Nonlinear evolution PDEs, fluid dynamics and
transport equations: theoretical foundations and applications and also by INdAM GNAMPA. LCB
also thanks King Abdullah University of Science and Technology (KAUST)
for the support and hospitality during the preparation of the paper. SG was partially supported by King Abdullah University of Science and Technology (KAUST) baseline funds. SG also acknowledges partial support from the Austrian Science Fund (FWF), grants P33010 and F65. This work has received funding from the European Research Council (ERC) under the European Union's Horizon 2020 research and innovation programme, ERC Advanced Grant no. 101018153.

\medskip

{\bf Conflict of interest.} The authors have no conflict of interest to report.
%
\def\ocirc#1{\ifmmode\setbox0=\hbox{$#1$}\dimen0=\ht0 \advance\dimen0
  by1pt\rlap{\hbox to\wd0{\hss\raise\dimen0
  \hbox{\hskip.2em$\scriptscriptstyle\circ$}\hss}}#1\else {\accent"17 #1}\fi}
  \def\polhk#1{\setbox0=\hbox{#1}{\ooalign{\hidewidth
  \lower1.5ex\hbox{`}\hidewidth\crcr\unhbox0}}} \def\cprime{$'$}

\end{document}